\documentclass[12pt]{amsart}
\usepackage{mathtools,amsthm,mathrsfs,txfonts,yfonts,titletoc}
\usepackage[all]{xy}
\usepackage[
	colorlinks,anchorcolor=blue,citecolor=blue,linkcolor=blue,urlcolor =blue,
	bookmarksdepth=2
]{hyperref}
\usepackage{fancyhdr}
\fancypagestyle{titlepage}
{
	\fancyhf{}

	\fancyfoot[l]{
		\href{https://mathscinet.ams.org/mathscinet/msc/msc2010.html}
	}
}

\newtheorem*{thm*}{Theorem}
\newtheorem*{prop*}{Proposition}
\newtheorem*{cor*}{Corollary}
\newtheorem{thm}{Theorem}[section]

\newtheorem{prop}[thm]{Proposition}
\newtheorem{cor}[thm]{Corollary}
\newtheorem{ques}[thm]{Question}
\newtheorem{lemma}[thm]{Lemma}

\numberwithin{equation}{section}

\newcommand{\Z}{\varmathbb{Z}}
\newcommand{\C}{\varmathbb{C}}
\newcommand{\fO}{\mathcal{O}}
\newcommand{\lP}{\varmathbb{P}}

\begin{document}
\title{\rm Cremona transformations and derived equivalences of K3 surfaces}
\author{Brendan Hassett
	\and Kuan-Wen Lai}
\address{Department of Mathematics, Brown University,
	Box 1917, 151 Thayer Street,
	Providence, RI 02912, USA}
\email{bhassett@math.brown.edu}
\email{kwlai@math.brown.edu}

\begin{abstract}
We exhibit a Cremona transformation of $\lP^4$ such that the base loci of the map and its inverse are birational to K3 surfaces. The two K3 surfaces are derived equivalent but not isomorphic to each other. As an application, we show that the difference of the two K3 surfaces annihilates the class of the affine line in the Grothendieck ring of varieties.
\end{abstract}

\maketitle

\thispagestyle{titlepage}

\section*{Introduction}
Let $X$ be a smooth complex projective variety that is \emph{rational}, i.e., admits a birational map $\varrho:\lP^r\dashrightarrow X$ where $r=\dim(X)$. The map $\varrho$ blows up various subvarieties of $\lP^r$ --- to what extent are these determined by $X$? We can always precompose $\varrho$ by a birational automorphism of $\lP^r$ (i.e., a \emph{Cremona transformation}) so we must take this into account.

For small dimensions these subvarieties are determined by $X$. When $r=1$, $\varrho$ extends to an isomorphism; if $r=2$, $\varrho$ is resolved by blowing up points in $\lP^2$. The case of threefolds was analyzed by Clemens and Griffiths \cite{CG72}. We may assume that $\varrho$ (resp.~$\varrho^{-1}$) is resolved by blowing up a finite number of points and nonsingular irreducible curves; let $C_1,\ldots,C_k$ (resp.~$D_1,\ldots,D_l$) denote those of positive genus. Comparing the Hodge structures on middle cohomology groups using the blow-up formula, we obtain an isomorphism of principally polarized abelian varieties:
\[
	J(C_1) \times \cdots \times J(C_k) \simeq J(X) \times J(D_1) \times \cdots \times J(D_l).
\]
The factors are Jacobians of curves and the intermediate Jacobian of $X$. Principally polarized abelian varieties admit unique decompositions into irreducible factors and the Jacobian of a curve is irreducible with respect to the natural polarization. It follows that $J(X)$ can be expressed as a product of Jacobians of curves $C_{i_1},\ldots,C_{i_t}$, $\{i_1,\ldots,i_t \} \subset \{1,\ldots,k\}$, and these curves are determined up to isomorphism by the Torelli Theorem. 

Therefore, we focus on fourfolds and their middle cohomology. Suppose that a smooth projective surface $\Sigma$ is contained in the base locus of $\varrho$. The blow-up formula gives a homomorphism of Hodge structures
\[
	\beta:H^2(\Sigma,\Z)(-1) \rightarrow H^4(X,\Z);
\]
can we recover $\Sigma$ from its image? Keeping track of divisor classes of $\Sigma$ is complicated, as they might disappear under subsequent blow-downs. Thus all we can expect to recover is the transcendental cohomology $T(\Sigma) \subset H^2(\Sigma,\Z)$. 

Mukai and Orlov \cite{Orl97} have shown that K3 surfaces offer many examples of non-birational surfaces $R$ and $\hat{R}$ with $T(R)\simeq T(\hat{R})$ as integral Hodge structures. These are explained through the notion of {\emph{derived equivalence}}. There are cubic fourfolds \cite{Has16} whose Hodge structures show the trace of several derived equivalent K3 surfaces. However, these are not known to be rational. Nevertheless,
this raises a question:

\begin{ques}
Let $R$ and $\hat{R}$ be derived equivalent K3 surfaces. Do there exist smooth projective fourfolds $X$, $P$, and $\hat{P}$ and birational maps
\[
	\varrho:P \dashrightarrow X,\quad\hat{\varrho}:\hat{P} \dashrightarrow X,
\]
such that $R$ and $\hat{R}$ are birational to components of the base loci of $\varrho$ and $\hat{\varrho}$ respectively, and the induced 
\[
	\beta:H^2(R,\Z)(-1) \rightarrow H^4(X,\Z), \quad\hat{\beta}:H^2(\hat{R},\Z)(-1) \rightarrow H^4(X,\Z)
\]
induce an isomorphism $T(R) \simeq T(\hat{R})$?
\end{ques}

In other words, are derived equivalences of K3 surfaces induced by birational maps? It makes sense to start with the case where $P\simeq \hat{P} \simeq \lP^4$. Are derived equivalences of K3 surfaces induced by Cremona transformations?

This last question may be too ambitious, as the base loci of Cremona transformations are highly constrained. According to Crauder and Katz \cite{CK89}, the Cremona transformation of $\lP^4$ which can be resolved by blowing up along a smooth and irreducible surface $S\subset\lP^4$ occurs as one of the following two cases:
\begin{enumerate}
	\item $S$ is a quintic elliptic scroll $S=\lP_C(E)$, where $C$ is an elliptic curve and $E$ is a rank two vector bundle with $e=-\deg(\bigwedge^2E)=-1$.
	\item $S$ is a degree 10 determinantal surface given by the vanishing of the $4\times4$ minors of a $4\times5$ matrix of linear forms.
\end{enumerate}

Here we present an example where derived equivalences of K3 surfaces are explained through Cremona transformations, and offer further evidence that such examples are quite rare. We can explain derived equivalences among degree $12$ K3 surfaces in this way; however, we do not know how to realize derived equivalences of higher degree K3 surfaces.

Our construction gives new examples of zero-divisors in the Grothendieck ring of complex algebraic varieties. The difference of each derived equivalent pair is non-vanishing in the ring and annihilated by the class of the affine line. The first example in this direction is given by the Pfaffian-Grassmannian Calabi-Yau threefolds \cite{Bor15, Mar16}. Other examples include Calabi-Yau threefolds from Grassmannians of type $G_2$ \cite{IMOU16G2, Kuz16}. Kuznetsov and Shinder \cite{KS16} have formulated general conjectures relating derived equivalence to zero-divisors in the Grothendieck ring; our example is an instance of \cite[Conj.~1.6]{KS16}. The relationship between zero-divisors in the Grothendieck ring and approaches to the rationality of cubic fourfolds is discussed in \cite{GS14}.

Section~\ref{sect:setup} presents preliminary results on Cremona transformations with singular base loci. The construction of our rational map is in Section~\ref{sect:const} and we verify the non-trivial derived equivalence in Section~\ref{sect:DE}. We apply the results to study the Grothendieck ring in Section~\ref{sect:Groth}. Section~\ref{sect:exclude} shows these constructions do not admit obvious extensions through a generalization of the classification of Crauder and Katz; the underlying computations are also used to analyze the maps defined in Section~\ref{sect:const}.

\medskip
\noindent {\bf Acknowledgments:}
We are grateful for conversations with Lev Borisov that inspired this work. The manuscript benefited from correspondence with Alexander Kuznetsov. The authors are grateful for the support of the National Science Foundation through DMS-1551514.

\section{Cremona transformation with singular base locus}\label{sect:setup}
\subsection{Terminology and notation}\label{subsect:neceCond}
A \emph{Cremona transformation} of $\lP^n$ is a birational map $f:\lP^n\dashrightarrow\lP^n$. Its \emph{base locus} ${\rm Bs}(f)$ is the subscheme where $f$ is undefined. 

Throughout this paper, we consider the Cremona transformation $f:\lP^4\dashrightarrow\lP^4$ with base locus resolved by blowing up an irreducible surface $S$, with singular locus consisting of \emph{transverse double points}, which means
a point where the surface has two smooth branches meeting transversally. 

Suppose $S$ has $\delta$ transverse double points which form a subset $\Delta\subset\lP^4$. The blowup of $\lP^4$ along $S$ can be factored as follows:
\begin{enumerate}
\item Blow up $\lP^4$ along $\Delta$, introducing $\delta$ exceptional divisors $E_1,...,E_\delta$ isomorphic to $\lP^3$. Let $P$ denote the resulting fourfold and $S'$ the proper transform of $S$, which is now smooth.
\item Blow up $P$ along $S'$ to obtain $P'$. Let $E$ denote the resulting exceptional divisor and $E_1',...,E_\delta'$ the proper transforms of the first group of exceptional divisors. Each $E_i'$ is isomorphic to $E_i\simeq\lP^3$ blown up along two skew lines $Q_i',Q_i''\subset E_i$.
\item\label{step3} Each $E_i'$ is a $\lP^1$-bundle over $\lP^1\times\lP^1$. Indeed, through each $p\in E_i$ not on $Q_i'$ and $Q_i''$ passes a unique line $l$ intersecting $Q_i'$ and $Q_i''$. The bundle map is given by $p\mapsto(l\cap Q_i',l\cap Q_i'')\in Q_i'\times Q_i''$. Blow down each $E_i'$ to $\lP^1\times\lP^1$. The resulting $X$ is isomorphic to ${\rm Bl}_S\lP^4$.
\end{enumerate}

\noindent{\bf Remark.} The blowup $X\rightarrow\lP^4$ has a quadric surface $Q_i$, $i=1,...\delta$, over each transverse double point of $S$. Then $P'$ is obtained as the blowup of $X$ along these quadrics.

Let $\pi_1:X\rightarrow\lP^4$ be the blowup along $S$ and $\pi_2:X\rightarrow\lP^4$ the resolution of $f$ so that $\pi_2=\pi_1\circ f$. We organize these maps into a diagram:
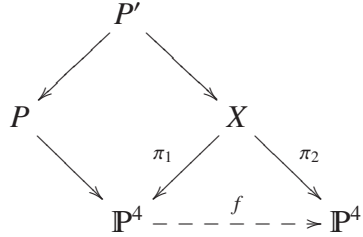
\[
\xymatrix{
	&P'\ar[dl]\ar[dr]&&\\
	P\ar[dr]&&X\ar_{\pi_1}[dl]\ar[dr]^{\pi_2}&\\
	&\lP^4\ar@{-->}[rr]^f&&\lP^4
}
\]
Note that, by the definition of blowup, $X$ is exactly the graph of $f$. Let $L$ (resp. $M$) denote the divisor of the hyperplane class of the left (resp. right) $\lP^4$. We also use 
$L$ (resp. $M$) to denote its pullbacks to $X$, $P$ and $P'$ (resp. $X$ and $P'$). 

It's clear that $L^4=1$. We have 
\begin{equation}\label{M4=1}
	M^4=1
\end{equation}
on $X$ as $f$ is birational. We define $n$ by 
\begin{equation}\label{L3M=n}
	L^3M=n
\end{equation}
and $\xi$ by 
\begin{equation}\label{LM3=xi}
	LM^3=\xi.
\end{equation}
We may interpret $n$ and $\xi$ as the degrees of the homogeneous forms inducing $f$ and $f^{-1}$ respectively. Define $m$ as the multiplicity of $S$ in the base locus. It is clear that
\[
	M=nL-m\left(E+2\Sigma_{i=1}^\delta E_i'\right)\quad\mbox{on}\; P'.
\]
Since a nondegenerate subvariety in projective space has degree greater than one, the linear system in $|M|$ inducing $P'\rightarrow\lP^4$ must be complete. Thus we have
\begin{equation}\label{h0M=5}
	h^0\left(P',M\right)=5.
\end{equation} 
We use these equations in our classification of Cremona
transformations below.

\subsection{Computing the intersection numbers}\label{subsect:intNum}
Let $\Sigma$ denote the normalization of $S$ and let $K_\Sigma$ be its canonical class. Then the blowup of $\Sigma$ along the preimage of the transverse double points is isomorphic to $S'$. We denote by $C$ a general sectional curve of $S$ and also its preimages in $\Sigma$ and $S'$. Let $d=C^2=\deg S$. Note that $E_i\cap S'=Q_i'\cup Q_i''$ are exactly the exceptional curves on $S'$ over the $i$-th double point.

\begin{lemma}\label{intEEi}
We have $LE_i'=0$. We also have $E^3E_i'=-4$, $E^2E_i'^2=2$, $EE_i'^3=0$ and $E_i'^4=-1$.
\end{lemma}
\begin{proof}
First, $LE_i'=0$ since their intersection is empty.

Recall that $E_i'$ is isomorphic to $E_i\simeq\lP^3$ blown up at skew lines $Q_i'$ and $Q_i''$. Write ${\rm Pic}(E_i')=\left<H,\widetilde{Q}',\widetilde{Q}''\right>$ where $H$ is the polarization from $\lP^3$ while $\widetilde{Q}'$ and $\widetilde{Q}''$ are the exceptional divisors over the lines. We clearly have $\widetilde{Q}'\widetilde{Q}''=0$ and $\widetilde{Q}'H^2=\widetilde{Q}''H^2=0$. Since $N_{Q_i'/\lP^3}=\fO_{Q_i'}(1)\oplus\fO_{Q_i'}(1)$ then writing $\zeta=c_1(\fO_{\lP(N_{Q_i'/\lP^3})}(1))$ we obtain $\zeta^2+2H\zeta=0$ in the Chow group of $\widetilde{Q}'=\lP(N_{Q_i'/\lP^3})$. We have $\widetilde{Q}'|_{\widetilde{Q}'}=-\zeta$ so that $$\widetilde{Q}'^2H=-\zeta H=-1, \quad\widetilde{Q}'^3=\zeta^2=-2H\zeta=-2.$$

We have $N_{E_i'/P'}=\fO(-H)$ and $E|_{E_i'}=\widetilde{Q}'+\widetilde{Q}''$. Thus we obtain
\[\begin{array}{rcl}
E^3E_i'&=&(\widetilde{Q}'+\widetilde{Q}'')^3=\widetilde{Q}'^3+\widetilde{Q}''^3=-4\\
E^2E_i'^2&=&(\widetilde{Q}'+\widetilde{Q}'')^2(-H)=2\\
EE_i'^3&=&(\widetilde{Q}'+\widetilde{Q}'')(-H)^2=0\\
E_i'^4&=&(-H)^3=-1.
\end{array}\]
\end{proof}

\begin{lemma}\label{intLE}
The intersection numbers involving $L$ and $E$ are
\begin{enumerate}
	\item $L^3E=0$ and $L^2E^2=-d$
	\item $LE^3=-5d-K_\Sigma C$
	\item\label{intLE3} $E^4=-15d-5K_\Sigma C-c_2(\Sigma)+6\delta$
	\setcounter{enumi}{2}\renewcommand\theenumi{\arabic{enumi}'}
	\item\label{intLE3'} $E^4=d^2-25d-10K_\Sigma C-{K_\Sigma}^2+4\delta$
\end{enumerate}
\end{lemma}
\begin{proof}
$L^3E=0$ since a general line doesn't intersect $S$. We have $L^2E^2=-\deg S'=-d$.

Assume that $C=S\cap L$ for some hyperplane $L\simeq\lP^3$. Then $LE^3=s(C,L)_0$ the zeroth Segre class of $C$ in $L$, which equals $[c(N_{C/L})^{-1}]_0=[c(C)c(\iota^*T_{\lP^3})^{-1}]_0=[([C]-K_\Sigma C-C^2)([C]-4d)]_0=-5d-K_\Sigma C$.

We have $E^4=-s(S',P)_0=-[c(N_{S'/P})^{-1}]_0=-[c(S')\left.c(P)\right|_{S'}^{-1}]_0$. Let $\epsilon:P\rightarrow\lP^4$ be the blowup. The blowup formula for Chern classes gives
\[\begin{array}{ccl}
	c(P)&=&\epsilon^*c(\lP^4)+(1+\Sigma_iE_i)(1-\Sigma_iE_i)^4-1\\
	&=&([P]+L)^5+\Sigma_i(-3E_i+2{E_i}^2+2{E_i}^3-3{E_i}^4).
\end{array}\]
Thus we have
\[\begin{array}{ccl}
	\left.c(P)\right|_{S'}&=&([S']+5C+10C^2)+\Sigma_i[-3(Q_i'+Q_i'')+2(Q_i'^2+Q_i''^2)]\\
	&=&[S']+5C-3\Sigma_i(Q_i'+Q_i'')+10d-4\delta
\end{array}\]
and also
\[
	\left.c(P)\right|_{S'}^{-1}=[S']-5C+3\Sigma_i(Q_i'+Q_i'')+15d-14\delta.
\]
Let $\tau:S'\rightarrow\Sigma$ be the blowup. Then we have
\[
	c(S')=[S']-\tau^*K_\Sigma-\Sigma_i(Q_i'+Q_i'')+c_2(\Sigma)+2\delta.
\]
Multiply the results to get $E^4=-15d-5K_\Sigma C-c_2(\Sigma)+6\delta$.

Another expression for $E^4$ is derived from
\[
	-[c(N_{S'/P})^{-1}]_0=c_2(N_{S'/P})-c_1(N_{S'/P})^2.
\]
We have $c_2(N_{S'/P})=d^2-4\delta$. On the other hand
\[\begin{array}{ccl}
	c_1(N_{S'/P})&=&\left.c_1(T_{P})\right|_{S'}-c_1(T_{S'})\\
	&=&-\left.(-5L+3\Sigma_iE_i)\right|_{S'}-(-K_{S'})\\
	&=&5C+\tau^*K_\Sigma-2\Sigma_i(Q_i'+Q_i''),
\end{array}\]
hence we deduce
\[
	c_1\left(N_{S'/P'}\right)^2=25d+10K_\Sigma C+{K_{\Sigma}}^2-8\delta
\]
and also 
\[\begin{array}{ccl}
	E^4&=&\left(d^2-4\delta\right)-\left(25d+10K_\Sigma C+{K_{\Sigma}}^2-8\delta\right)\\
	&=&d^2-25d-10K_\Sigma C-{K_\Sigma}^2+4\delta.
\end{array}\]
\end{proof}

\section{Construction of our example}\label{sect:const}
In this section, we use Mukai's construction \cite{Muk88} to produce an explicit example of a degree 12 K3 surface $R\subset\lP^7$ together with three points $p_1,p_2,p_3\in R$. This example helps us prove the following theorem:

\begin{thm} \label{thm:existence}
Let $R\subset \lP^7$ be a generic K3 surface of degree $12$
and $\Pi:=\{p_1,p_2,p_3\} \subset R$ a generic triple of points.
\begin{enumerate}
\item\label{exist1} projection from $\Pi$ maps $R$ to a surface $S\subset \lP^4$
with three transverse double points;
\item\label{exist2} the complete linear system $M$ of quartics vanishing along $S$ 
cuts out $S$ scheme-theoretically;
\item\label{exist3} $M$ induces a birational map $f:\lP^4 \dashrightarrow\lP^4$;
\item\label{exist4} the base locus of the inverse $f^{-1}$ is also a projection of a degree $12$ K3 surface from three points.
\end{enumerate}
\end{thm}

\subsection{Orthogonal Grassmannian}
\label{subsect:OG}
Let $V$ be a 10-dimensional vector space equipped with a nondegenerate quadratic form $q$. The 5-dimensional subspaces of $V$ isotropic with respect to $q$ form a subvariety $\mathcal{S}$ of the Grassmannian G$(5,V)$. It has two components $\mathcal{S}^+$ and $\mathcal{S}^-$ which are isomorphic to each other. They are called \emph{orthogonal Grassmannians} and are denoted by OG$(5,V)$. 

Fix a 5-subspace $W\in{\rm OG}(5,V)$ and let $W^*$ be its orthogonal complement with respect to $q$. Then OG$(5,V)$ can be identified scheme theoretically as the zero locus in
\[
	\lP(\C\oplus\bigwedge^2 W\oplus \bigwedge^4W)\simeq\lP^{15}
\]
of the quadratic form \cite[\S 2]{IM04}
\begin{equation}\label{OG(5,V)}
\begin{array}{ccc}
\det W\oplus\bigwedge^2W\oplus W^*&\longrightarrow&\bigwedge^4W\oplus W\\
(x,\Omega,v)&\longmapsto&(x(v)+\frac{1}{2}\Omega\wedge\Omega,\,\Omega(v)).
\end{array}
\end{equation}
Here we choose an isomorphism $\C\simeq\det W$. This induces an isomorphism $\bigwedge^4W\simeq W^*$.

Let ${\bf x}=(x_0,...,x_{15})$ be the homogeneous coordinate for $\lP^{15}$. Then (\ref{OG(5,V)}) can be explicitly written down as ten quadrics:
\begin{equation}\label{OG(5,V)Ideal}
\begin{array}{l|l}
x_0x_{11}+x_5x_{10}-x_6x_9+x_7x_8 &  -x_1x_{12}+x_2x_{13}-x_3x_{14}+x_4x_{15}\\
x_0x_{12}+x_2x_{10}-x_3x_9+x_4x_8 & x_1x_{11}-x_5x_{13}+x_6x_{14}-x_7x_{15}\\
x_0x_{13}+x_1x_{10}-x_3x_7+x_4x_6 & -x_2x_{11}+x_5x_{12}-x_8x_{14}+x_9x_{15}\\
x_0x_{14}+x_1x_9-x_2x_7+x_4x_5 & x_3x_{11}-x_6x_{12}+x_8x_{13}-x_{10}x_{15}\\
x_0x_{15}+x_1x_8-x_2x_6+x_3x_5 & -x_4x_{11}+x_7x_{12}-x_9x_{13}+x_{10}x_{14}.
\end{array}
\end{equation}

\subsection{An explicit example}
Mukai \cite[\S3]{Muk88} proves that a generic K3 surface of degree 12 appears as a linear section of OG$(5,V)$ and vice versa.

For example, the $\lP^7\subset\lP^{15}$ spanned by the rows of the $8\times16$ matrix
\[{\bf H}=\left(
\footnotesize\arraycolsep=4pt
\begin{array}{rrrrrrrrrrrrrrrr}
-1& 3& 2& 0& 2&-3&-1& 0& 3& 1& 0& 3& 0& 2& 0&-3\\
 1& 0&-3& 0&-2& 1& 1& 0&-2&-1&-1&-1& 0& 4& 0& 2\\
-1&-3&-2& 0&-3& 0& 3& 2&-1&-3&-1& 2&-1& 2& 0& 3\\
 3& 0& 0& 2& 2& 3& 0& 1& 2&-1& 0& 2&-1&-2& 2& 3\\
 0&-1& 1&-1& 0& 1&-3& 3& 2& 2& 1& 3& 0&-3& 0&-3\\
 1& 0& 0& 0& 0&-1& 0& 1& 0& 0& 0& 0& 0& 0& 0& 0\\
 2& 1& 0& 1& 0& 0& 0& 1& 0& 0& 1& 0& 0& 0& 0& 0\\
 3& 0& 1& 0& 1& 0& 0& 0& 0& 0& 0& 0& 0& 0& 0& 0
\end{array}
\right)\]
cuts out a degree 12 K3 surface R on OG$(5,V)$. More explicitly, let ${\bf z}=(z_0,...,z_7)$ be homogeneous coordinates for $\lP^7$. We define the inclusion $\iota:\lP^7\hookrightarrow\lP^{15}$ by
\[
	{\bf x} = {\bf z}\cdot{\bf H}.
\]
Then we get $R=\iota^{-1}({\rm OG}(5,V))$.

The last three rows of {\bf H} are chosen as solutions of (\ref{OG(5,V)Ideal}) so that they form a triple of points $\Pi=\{p_1,p_2,p_3\}\subset R$. With this choice the projection from $\Pi$ is exactly the map
\[\begin{array}{cccc}
	\pi:&\lP^7&\dashrightarrow&\lP^4\\
	&(z_0,...,z_7)&\mapsto&(z_0,...,z_4)
\end{array}\]
which takes $R$ to $S=\pi(R)$.

We manipulate this example in a computer algebra system\footnote{The main program we use in this work is {\sc Singular} \cite{DGPS}} over the finite field $\varmathbb{F}_7$. We compute that $S$ is singular along three transverse double points $\{a_1,a_2,a_3\}$ and is the base locus of a Cremona transformation
\[
	\overline{f}:\lP^4\dashrightarrow\lP^4.
\]
Moreover, the base locus of the inverse $(\overline{f})^{-1}$ is again a surface $\overline{T}$ singular along three transverse double points $\{b_1,b_2,b_3\}$. The matrix $\bf H$ is chosen such that the preimage of $\{a_1,a_2,a_3\}$ on $R$ and the preimage of $\{b_1,b_2,b_3\}$ on the normalization of $\overline{T}$ are $\varmathbb{F}_7$-rational points. This is the smallest field where our computer could quickly find such an $\bf H$.

\subsection{Proof of Theorem~\ref{thm:existence}}\label{subsect:existence}
We prove Theorem~\ref{thm:existence} for our example first.

We confirm the following properties by computer over $\varmathbb{F}_7$:
\begin{enumerate}
\item $S$ is singular along three points. The preimage of each singular point on $R$ has two points outside $\Pi$. So they are transverse double points.
\item The ideal of $S$ is generated by five quartics $\overline{f}_0,...,\overline{f}_4$.
\end{enumerate}

The double-point formula \cite[Thm. 9.3]{Ful98} indicates that the three transverse double points of (\ref{exist1}) exist over characteristic zero. Indeed, let $\epsilon:\Sigma\rightarrow S$ be the normalization. Then $\Sigma$ is isomorphic to $R$ blown up at three points. The {\it double-point class} $\varmathbb{D}(\epsilon)\in{\rm CH_0}(\Sigma)$ is given by the formula
\[\begin{array}{ccl}
	\varmathbb{D}(\epsilon) &=& \epsilon^*\epsilon_*[\Sigma]-[\epsilon^*c(\lP^4)\cdot c(\Sigma)^{-1}]_0\\
	&=& \left<S,S\right>_{\lP^4} - (\epsilon^*c_2(\lP^4)-\epsilon^*c_1(\lP^4)\cdot c_1(\Sigma)-c_1(\Sigma)^2+c_2(\Sigma)).
\end{array}\]
It's easy to verify that $\varmathbb{D}(\epsilon)=6$. The quantity $\frac{1}{2}\varmathbb{D}(\epsilon)=3$ counts the number of singularities on $S$ with multiplicity if the singular locus is a finite set. Therefore (\ref{exist1}) implies that the singular locus of $S$ consists of three transverse double points. This proves Theorem~\ref{thm:existence}(\ref{exist1}).

The five quartics $\overline{f}_0,...,\overline{f}_4$ lift to a basis $f_0,...,f_4$ for the ideal of $S$ over characteristic zero. In particular Theorem~\ref{thm:existence}(\ref{exist2}) holds. The forms $f_0,...,f_4$ define a rational map
\[
	f:=(f_0,...,f_4):\lP^4\dashrightarrow\lP^4
\]
which reduces to
\[
	\overline{f}:=(\overline{f}_0,...,\overline{f}_4):\lP^4\dashrightarrow\lP^4
\]
over $\varmathbb{F}_7$. The degree of $f$ is computed by the self-intersection $M^4$, which can be expanded as the right-hand side of equation (\ref{1=M4}). It's easy to check that our example satisfies
\[
	(n,m,d,\delta) = (4,1,9,3),\quad K_\Sigma C=3\quad\mbox{and}\quad c_2(\Sigma)=27.
\]
Inserting these data into (\ref{1=M4}) we get $M^4=1$, i.e. the map $f$ is birational. Thus Theorem~\ref{thm:existence}(\ref{exist3}) holds.

The inverse $(\overline{f})^{-1}$ can be calculated by computer. It consists of five quartics also and the base locus is a surface $\overline{T}$ singular along three points. These are transverse double points since each point has two preimage points on the normalization. By the same reasons as above, the base locus of $f^{-1}$ is again a surface cut by five quartics and singular along three transverse double points. Then Theorem~\ref{thm:existence}(\ref{exist4}) follows from Theorem~\ref{thm:uniqueness}.

Next we prove Theorem~\ref{thm:existence} in the generic case.

It's clear that (\ref{exist1}), (\ref{exist2}) and (\ref{exist3}) of the theorem are open conditions, so they hold for a generic example. As a consequence of Theorem~\ref{thm:uniqueness}, property (\ref{exist4}) holds once Bs($f^{-1}$) is a surface cut out by five quartics and singular along three transverse double points. These are open conditions again so Theorem~\ref{thm:existence} holds for a generic example.

\subsection{Some geometry of the construction}\label{subsect:someGeo}
Let $f$ be a Cremona transformation of Theorem \ref{thm:existence}.
It has a resolution
\[
	\xymatrix{
	&X\ar_{\pi_1}[dl]\ar[dr]^{\pi_2}&\\
	\lP^4\ar@{-->}[rr]^f&&\lP^4.
}
\]
Let $S_L$ and $S_M$ be the base locus of $f$ and its inverse $f^{-1}$, respectively.
Then $\pi_1$ is the blowup along $S_L$.
Recall that $X$ coincides with the graph of $f$ as well as $f^{-1}$.
Therefore, $\pi_2$ is the blowup along $S_M$.

This example has $d = 9$, $K_\Sigma C = 3$, $c_2(\Sigma) = 27$, $\delta = 3$ and
\[
	M=4L-E-2\Sigma_{i=1}^3E_i'.
\]
Evaluating Lemmas \ref{intEEi} and \ref{intLE} with this data yields
\begin{cor}\label{cor:intNum}
	We have
	\begin{enumerate}
		\item $LE_i'=0$, $E^3E_i'=-4$, $E^2E_i'^2=2$, $EE_i'^3=0$, $E_i'^4=-1$,
		\item $L^3E = 0$, $L^2E^2 = -9$, $LE^3 = -48$, $E^4 = -159$.
	\end{enumerate}
	Thus consequently,
	\begin{enumerate}\setcounter{enumi}{2}
		\item $L^3M = 4$, $L^2M^2 = 7$, $LM^3 = 4$, $M^4 = 1$.
	\end{enumerate}
\end{cor}

Let $X_0,...,X_4$ be the homogeneous coordinates for $\lP^4$.
The Cremona transformation $f$ is ramified along the locus $\Theta$ where the Jacobian matrix
\[
	{\bf D}f=\left(\frac{\partial f_i}{\partial X_j}\right)_{5\times5}
\]
is degenerate.
So $\Theta$ is a degree 15 hypersurface in $\lP^4$ defined by
\[
	\det({\bf D}f)=0.
\]
This locus is called \emph{P-locus},
which is classically defined as the image of the exceptional divisor of the blowup $\pi_2$ \cite[\S 7.1.4]{Dol12}.
In particular, $\Theta$ is irreducible.
It also follows that $\Theta$ is the locus contracted by $f$ and its image is the base locus $S_M$.

\begin{prop}
	The locus $\Theta\subset\lP^4$ contracted by $f$ is an irreducible hypersurface of degree 15. 
	It has multiplicity four along $S_L$.
	Moreover, it equals the union of all of the 4-secant lines to $S_L$.
	The analogous statement holds for the inverse $f^{-1}$ by symmetry.
\end{prop}
\begin{proof}
Let $m$ be the multiplicity of $\Theta$ along $S_L$.
Then the divisor class of its pullback to $X$ equals
\[
	\Theta_X=15L-mE_X.
\]
Here we use $E_X$ to denote the exceptional divisor of the blowup $\pi_1$.
Because $\Theta$ is contracted onto a surface,
we have
\[
	0 = M^3\Theta_X
	= M^3(15L-mE_X)
	= 60-mM^3E_X.
\]
By definition, $E_X$ is mapped onto the $P$-locus of the inverse map $f^{-1}$.
In particular, it is again a degree 15 hypersurface in $\lP^4$ by symmetry.
So $M^3E_X=15$, which implies $m=4$.

Let $F_X$ be the exceptional locus of the blowup $\pi_2$.
We have $L=4M-F_X$ by symmetry, hence
\[
	F_X = 4M-L.
\]
(Note that this equals $15L-4E_X=\Theta_X$)
The fiber of the map $F_X\rightarrow S_M$ over a smooth point is represented by the class
\[
	l_X=\frac{F_XM^2}{\deg S_M}=\frac{1}{9}F_XM^2.
\]
The image $l=\pi_1(l_X)$ is a rational curve of degree
\[
	L\cdot l_X= \frac{1}{9}LF_XM^2 = \frac{1}{9}L(4M-L)M^2 =  \frac{1}{9}(16-7) = 1.
\]
The intersection number between $l$ and $S_L$ can be computed by
\[\begin{array}{ccl}
	E_X\cdot l_X&=&\frac{1}{9}E_XF_XM^2 = \frac{1}{9}(4L-M)(4M-L)M^2\\
	&=&\frac{1}{9}(64-28-4+4) = 4.
\end{array}\]
Hence the fibers of $F_X\rightarrow S_M$ away from the double points is mapped by $\pi_1$ to 4-secant lines to $S_L$.
In other words,
$S_L$ admits a family of 4-secant lines parametrized by the smooth locus of $S_M$.

Conversely, every 4-secant line $l$ to $S_L$ satisfies
\[
	l\cdot M = l\cdot (4L-E_X) = 4-4=0.
\]
So $l$ is contracted to a point by $f$.
Hence the union of the 4-secant lines to $S_L$ forms a 3-fold contained in $\Theta$ and thus coincides with $\Theta$.
\end{proof}

\section{Derived equivalences of K3 surfaces}\label{sect:DE}
Let's keep the notation of Section \ref{subsect:someGeo}.
By Theorem \ref{thm:existence},
there exists two degree 12 K3 surfaces $R_L$ and $R_M$ projected onto the base loci $S_L$ and $S_M$, respectively.
This section is devoted to the following:

\begin{thm}\label{thm:DE}
The two K3 surfaces $R_L$ and $R_M$ are derived equivalent. They are non-isomorphic if they have Picard number one.
\end{thm}

\begin{cor} \label{cor:hilbthree}
There is a birational map $\sigma:R_L^{[3]}\dashrightarrow R_M^{[3]}$ between the Hilbert schemes of length three subschemes.
\end{cor}

\subsection{Derived equivalences and general strategy} 
\label{subsect:DES}
Let $R$ and $\hat{R}$ denote K3 surfaces and $T(R)$ and $T(\hat{R})$ the corresponding transcendental lattices. Recall that $R$ and $\hat{R}$ are derived equivalent if and only if $T(R)$ and $T(\hat{R})$ are Hodge isometric \cite{Orl97}. Suppose that $R$ has Picard rank one and degree $2n$. Let $\tau(n)$ be
the number of prime factors of $n$. Then the number of isomorphism classes of K3 surfaces derived equivalent to $R$ is equal to $2^{\tau(n)-1}$ \cite{HLOY03}. Thus a general degree $12$ K3 surface admits a unique such partner.

Our general approach is to prove that $T(R_L)$ is isometric to $T(R_M)$ by showing that both of them can be identified as the transcendental sublattice of $H^4(X,\Z)$. Then we show that the induced isomorphism on the discriminant groups is nontrivial, which implies that $R_L$ and $R_M$ are not isomorphic to each other.

\subsection{The middle cohomology of $X$}
Retain the notation of Section~\ref{sect:setup}. Let $H_L$ be the polarization of $R_L$. Let $F_1$, $F_2$ and $F_3$ be the exceptional curves from the projection $R_L\dashrightarrow S_L$. We consider $H_L$, $F_1$, $F_2$ and $F_3$ as curves on $S_L$. Their strict transforms $\widetilde{H}_L$, $\widetilde{F}_1$, $\widetilde{F}_2$, $\widetilde{F}_3$ on $X$ together with $L^2$ and the quadrics $Q_1$, $Q_2$, $Q_3$ form a rank 8 sublattice $A_L(X)\subset H^4(X,\Z)$. We have
\[
	\widetilde{H}_L^2=-H_L^2=-12,\quad\widetilde{F}_i^2=-F_i^2=1\quad\mbox{and}\quad Q_i^2=-E_i'^4=1
\] 
where $i=1,2,3$. These classes are mutually disjoint, so the intersection matrix for $A_L(X)$ is
\begin{equation}\label{intAL}\begin{array}{c|cccc}
&L^2&\widetilde{H}_L&\widetilde{F}_{1,2,3}&Q_{1,2,3}\\
\hline
L^2&1&&&\\
\widetilde{H}_L&&-12&&\\
\widetilde{F}_{1,2,3}&&&I_{3\times3}&\\
Q_{1,2,3}&&&&I_{3\times3}\\
\end{array}\end{equation}
where $I_{3\times3}$ is the identity matrix of rank 3.

\begin{lemma}\label{decH4X}
There is a decomposition
\[
	H^4(X,\Z)\simeq H^4(X,\Z)_{alg}\oplus_\perp T(R_L)(-1).
\]
where $H^4(X,\Z)_{alg}$ is the sublattice spanned by algebraic classes. We have
\[
	H^4(X,\Z)_{alg}=A_L(X)
\]
when $R_L$ has Picard number one. Here we use $\Lambda(-1)$ to denote a lattice $\Lambda$ equipped with the negative of its original product.
\end{lemma}
\begin{proof}
We apply the blowup formula for cohomology to the composition $P'\rightarrow P\rightarrow\lP^4$ and the map $P'\rightarrow X$ to obtain two decompositions for $H^4(P',\Z)$. Then we compare them to get our result.

Let $S_L'\subset P$ be the strict transform of $S_L$. Recall that $S_L'$ is isomorphic to $R_L$ blown up at 3+6=9 points, where 3 are from the projection $R_L\dashrightarrow S_L$ while 6 are from the resolution $S_L'\rightarrow S_L$. Thus we have
\[
	H^2(S_L',\Z)\simeq\left<F_i,\;Q_i',\;Q_i''\right>_{i=1,2,3}\oplus H^2(R_L,\Z).
\]
Let $\widetilde{Q}_i'$ and $\widetilde{Q}_i''$ be the strict transforms of $Q_i'$ and $Q_i''$ on $P'$. Since $P'\rightarrow P$ is the blowup along $S_L'$, we have
\begin{equation}\label{decompP1}\begin{array}{ccl}
	H^4(P',\Z)&\simeq&H^4(P,\Z)\oplus H^2(S_L',\Z)(-1)\\
	&\simeq&\left<L^2,\;E_i'^2,\;\widetilde{F}_i,\;\widetilde{Q}_i',\;\widetilde{Q}_i''\right>_{i=1,2,3}\oplus H^2(R_L,\Z)(-1).
\end{array}\end{equation}

For every $i$, we have
\[
\widetilde{Q}_i'^2=-Q_i'^2=1,\quad\widetilde{Q}_i''^2=-Q_i''^2=1,\quad E_i'^2\widetilde{Q}_i'=E_i'^2\widetilde{Q}_i''=0
\]
and $E_i'^4=-1$. With these it's straightforward to prove the isometry
\[
	\left<E_i'^2,\;\widetilde{Q}_i',\;\widetilde{Q}_i''\right>\simeq
	\left<E_i'^2+\widetilde{Q}_i',\;E_i'^2+\widetilde{Q}_i'',\;E_i'^2+\widetilde{Q}_i'+\widetilde{Q}_i''\right>,
\]
whence (\ref{decompP1}) equals
\begin{equation}\label{decompP2}\begin{array}{cl}
	H^4(P',\Z)\simeq&\left<E_i'^2+\widetilde{Q}_i',\;E_i'^2+\widetilde{Q}_i'',\;E_i'^2+\widetilde{Q}_i'+\widetilde{Q}_i'',\;L^2,\;\widetilde{F}_i\right>_{i=1,2,3}\\
	&\oplus H^2(R_L,\Z)(-1).
\end{array}\end{equation}

By the description of the map $E_i\dashrightarrow Q_i$, the two fiber classes on $Q_i\simeq\lP^1\times\lP^1$ pullback to hyperplanes in $E_i$ containing either $Q_i'$ or $Q_i''$, which correspond to the classes $-E_i'^2-\widetilde{Q}_i'$ or $-E_i'^2-\widetilde{Q}_i''$ on $P'$, respectively. The map $P'\rightarrow X$ is the blowup along $Q_i$, $i=1,2,3$, so
\begin{equation}\label{decompX}
	H^4(P',\Z)\simeq\left<E_i'^2+\widetilde{Q}_i',\;E_i'^2+\widetilde{Q}_i''\right>_{i=1,2,3}\oplus H^4(X,\Z)
\end{equation}
Combining (\ref{decompP2}) and (\ref{decompX}) we get
\[\begin{array}{ccl}
	H^4(X,\Z)&\simeq&\left<E_i'^2+\widetilde{Q}_i',\;E_i'^2+\widetilde{Q}_i''\right>_{i=1,2,3}^\perp\\
	&\simeq&\left<E_i'^2+\widetilde{Q}_i'+\widetilde{Q}_i'',\;L^2,\;\widetilde{F}_i\right>_{i=1,2,3}\oplus H^2(R_L,\Z)(-1).
\end{array}\]

Both $Q_i$ and $E_i'^2+\widetilde{Q}_i'+\widetilde{Q}_i''$ are orthogonal to $L^2$, $\widetilde{F}_{i=1,2,3}$ and $H^2(R_L,\Z)$, and $Q_i^2=(E_i'^2+\widetilde{Q}_i'+\widetilde{Q}_i'')^2=1$, so $Q_i=\pm(E_i'^2+\widetilde{Q}_i'+\widetilde{Q}_i'')$. Therefore
\[\begin{array}{ccl}
	H^4(X,\Z)&\simeq&\left<Q_i,\;L^2,\;\widetilde{F}_i\right>_{i=1,2,3}\oplus H^2(R_L,\Z)(-1)\\
	&\simeq&\left<Q_i,\;L^2,\;\widetilde{F}_i\right>_{i=1,2,3}\oplus NS(R_L)(-1)\oplus_\perp T(R_L)(-1)\\
	&\simeq&H^4(X,\Z)_{alg}\oplus_\perp T(R_L)(-1)
\end{array}\]
where $NS(R_L)$ is the N\'{e}ron-Severi lattice of $R_L$.

When $R_L$ has Picard number one, we have $NS(R_L)(-1)\simeq\left<\widetilde{H}_L\right>$. In this case
\[
	H^4(X,\Z)_{alg}\simeq\left<Q_i,\;L^2,\;\widetilde{F}_i,\;\widetilde{H}_L\right>_{i=1,2,3}=A_L(X).
\]
\end{proof}

Lemma~\ref{decH4X} also proves the decomposition
\[
	H^4(X,\Z)\simeq H^4(X,\Z)_{alg}\oplus_\perp T(R_M)(-1)
\]
from the side of $f^{-1}$. So there is an isometry
\[
	T(R_L)\simeq H^4(X,\Z)_{alg}^\perp(-1)\simeq T(R_M)
\]
which allows us to conclude that
\begin{prop}\label{deEquiv}
$R_L$ and $R_M$ are derived equivalent.
\end{prop}

\subsection{The discriminant groups}
For an arbitrary lattice $\Lambda$ with dual lattice $\Lambda^*:={\rm Hom}(\Lambda,\Z)$, we denote by ${\rm d}\Lambda:=\Lambda^*/\Lambda$ its discriminant group. 

Let $A_M(X)$ be the lattice constructed in the same way as $A_L(X)$ from the side of $f^{-1}$. Assume $R_L$ and $R_M$ have Picard number one. Then Lemma~\ref{decH4X} implies that there is an isometry
\[
	\varphi:A_M(X)\oplus_\perp T(R_M)(-1)\xrightarrow{\sim}A_L(X)\oplus_\perp T(R_L)(-1)
\]
such that $\varphi=\varphi_A\oplus\varphi_T$ with respect to the decompositions. It induces the commutative diagram
\[
\xymatrix{
	\mathrm{d}A_M(X) \ar[r]^{\varphi_{A_*}}_{\sim} \ar[d]_{\sim}  &
 	\mathrm{d}A_L(X) \ar[d]_{\sim} \\
	\mathrm{d}T(R_M) \ar[r]^{\varphi_{T_*}}_{\sim} & \mathrm{d}T(R_L).
}
\]
These groups are all isomorphic to $\Z/12\Z$. From the intersection matrix (\ref{intAL}) we know that ${\rm d}A_L(X)$ is generated by $-\widetilde{H}_L/12$. Similarly, ${\rm d}A_M(X)$ is generated by $-\widetilde{H}_M/12$ where $H_M$ is the polarization of $R_M$ and $\widetilde{H}_M$ is the strict transform on $X$.

\begin{lemma}\label{baseChg}
We have the following equations in $H^4(X,\Z)$
	\begin{enumerate}
	\item $M^2=7L^2-3\widetilde{H}_L+4(\widetilde{F}_1+\widetilde{F}_2+\widetilde{F}_3)+2(Q_1+Q_2+Q_3)$
	\item $\widetilde{H}_M=36L^2-17\widetilde{H}_L+24(\widetilde{F}_1+\widetilde{F}_2+\widetilde{F}_3)+12(Q_1+Q_2+Q_3)$
	\end{enumerate}
\end{lemma}
\begin{proof}
The following computation is based on Corollary \ref{cor:intNum}.

Assume that
\[
	M^2=aL^2+b\widetilde{H}_L+f_1\widetilde{F}_1+f_2\widetilde{F}_2+f_3\widetilde{F}_3+g_1Q_1+g_2Q_2+g_3Q_3.
\]
Then $a=L^2M^2=7$. For $i=1,2,3$, we have
\[\begin{array}{l}
	g_i = M^2Q_i = -M^2E_i'^2\\
	= -(4L-E-2\Sigma_jE_j')^2E_i'^2 = -(-E-2\Sigma_jE_j')^2E_i'^2\\
	= -E^2E_i'^2-4EE_i'^3-4E_i'^4 = -2-0+4 = 2.
\end{array}\]
Let $\widetilde{C}$ be the strict transform of the sectional curve $C$ on $P'$. Note that
\[
	LM = 4L^2-LE=4L^2-\widetilde{C}=4L^2-\widetilde{H}_L+\Sigma_i\widetilde{F}_i,
\]
so we find
\[\begin{array}{l}
	4 = LM^3 = (LM)M^2\\
	= (4L^2-\widetilde{H}_L+\Sigma_i\widetilde{F}_i)(7L^2+b\widetilde{H}_L+\Sigma_if_j\widetilde{F}_j+2\Sigma_kQ_k)\\
	= 28+12b+f_1+f_2+f_3
\end{array}\]
and thus
\begin{equation}\label{Mf}
	f_1+f_2+f_3 = -12b-24.
\end{equation}
We also have
\[\begin{array}{l}
	1 = M^4 = (7L^2+b\widetilde{H}_L+\Sigma_if_j\widetilde{F}_j+2\Sigma_kQ_k)^2\\
	= 49-12b^2+f_1^2+f_2^2+f_3^2+12
\end{array}\]
which is equivalent to
\begin{equation}\label{Mf2}
	f_1^2+f_2^2+f_3^2 = 12b^2-60.
\end{equation}
By the Cauchy-Schwarz inequality
\begin{equation}\label{MCSineq}\begin{array}{l}
	(f_1+f_2+f_3)^2=((1,1,1)\cdot(f_1,f_2,f_3))^2\\
	\leq(1,1,1)^2(f_1,f_2,f_3)^2=3(f_1^2+f_2^2+f_3^2).
\end{array}\end{equation}
Applying (\ref{Mf}) and (\ref{Mf2}) we get $(-12b-24)^2\leq3(12b^2-60)$, i.e.
\[
	3b^2+16b+21=(3b+7)(b+3)\leq0.
\]
The only integer solution is $b=-3$. Because (\ref{MCSineq}) becomes an equality in this case, we have $(f_1,f_2,f_3)=f(1,1,1)$ for some integer $f$. We obtain $f=4$ by setting $b=-3$ in (\ref{Mf}). As a result, we find
\[
	M^2=7L^2-3\widetilde{H}_L+4(\widetilde{F}_1+\widetilde{F}_2+\widetilde{F}_3)+2(Q_1+Q_2+Q_3).
\]

Next, assume that
\[
	\widetilde{H}_M=aL^2+b\widetilde{H}_L+f_1\widetilde{F}_1+f_2\widetilde{F}_2+f_3\widetilde{F}_3+g_1Q_1+g_2Q_2+g_3Q_3.
\]
By symmetry, $\widetilde{H}_LL^nM^{2-n}=\widetilde{H}_MM^nL^{2-n}$ for $n=0,1,2$. In particular,
\[
	a=\widetilde{H}_ML^2=\widetilde{H}_LM^2=-3\widetilde{H}_L^2=36.
\]
We have
\[\begin{array}{l}
	12=\widetilde{H}_L(4L^2-\widetilde{H}_L+\Sigma_i\widetilde{F}_i)=\widetilde{H}_L(LM)=\widetilde{H}_M(ML)\\
	=(36L^2+b\widetilde{H}_L+\Sigma_if_i\widetilde{F}_i+\Sigma_jg_jQ_j)(4L^2-\widetilde{H}_L+\Sigma_i\widetilde{F}_i)\\
	=144+12b+f_1+f_2+f_3.
\end{array}\]
Rearrange to obtain
\begin{equation}\label{Hf}
	f_1+f_2+f_3 = -12b-132.
\end{equation}
Applying the symmetry again, we get
\[\begin{array}{l}
	0=\widetilde{H}_LL^2=\widetilde{H}_MM^2\\
	=(36L^2+b\widetilde{H}_L+\Sigma_if_i\widetilde{F}_i+\Sigma_jg_jQ_j)(7L^2-3\widetilde{H}_L+4\Sigma_i\widetilde{F}_i+2\Sigma_jQ_j)\\
	=252+36b+4(f_1+f_2+f_3)+2(g_1+g_2+g_3)
\end{array}\]
whence
\[
	2(f_1+f_2+f_3)+(g_1+g_2+g_3)=-18b-126
\]
and combining with (\ref{Hf}) gives
\begin{equation}\label{Hg}
	g_1+g_2+g_3 = 6b+138.
\end{equation}
We also have
\[\begin{array}{l}
	-12=\widetilde{H}_L^2=\widetilde{H}_M^2=(36L^2+b\widetilde{H}_L+\Sigma_if_i\widetilde{F}_i+\Sigma_jg_jQ_j)^2\\
	=1296-12b^2+f_1^2+f_2^2+f_3^2+g_1^2+g_2^2+g_3^2,
\end{array}\]
from which we obtain
\begin{equation}\label{Hf2g2}
	f_1^2+f_2^2+f_3^2+g_1^2+g_2^2+g_3^2=12b^2-1308.
\end{equation}
By the Cauchy-Schwarz inequality,
\begin{equation}\label{Hfineq}\begin{array}{l}
	(-12b-132)^2=(f_1+f_2+f_3)^2=((1,1,1)\cdot(f_1,f_2,f_3))^2\\
	\leq(1,1,1)^2(f_1,f_2,f_3)^2=3(f_1^2+f_2^2+f_3^2)
\end{array}\end{equation}
and
\begin{equation}\label{Hgineq}\begin{array}{l}
	(6b+138)^2=(g_1+g_2+g_3)^2=((1,1,1)\cdot(g_1,g_2,g_3))^2\\
	\leq(1,1,1)^2(g_1,g_2,g_3)^2=3(g_1^2+g_2^2+g_3^2).
\end{array}\end{equation}
Add the two inequalities and then apply (\ref{Hf2g2}) to get
\begin{equation}\label{Hf2g2ineq}\begin{array}{l}
	(-12b-132)^2+(6b+138)^2\\
	\leq3(f_1^2+f_2^2+f_3^2+g_1^2+g_2^2+g_3^2)=3(12b^2-1308)
\end{array}\end{equation}
which can be arranged as
\[
	2b^2+67b+561=(2b+33)(b+17)\leq0.
\]
The only integer solution is $b=-17$ which makes (\ref{Hf2g2ineq}) an equality. This forces (\ref{Hfineq}) and (\ref{Hgineq}) to be equalities also. Therefore $(f_1,f_2,f_3)=f(1,1,1)$ and $(g_1,g_2,g_3)=g(1,1,1)$ for some integers $f$ and $g$. We get $f=24$ from (\ref{Hf}) and $g=12$ from (\ref{Hg}). As a consequence,
\[
	\widetilde{H}_M=36L^2-17\widetilde{H}_L+24(\widetilde{F}_1+\widetilde{F}_2+\widetilde{F}_3)+12(Q_1+Q_2+Q_3).
\]
\end{proof}

\begin{prop}\label{times7}
The isomorphism ${\varphi_A}_*:{\rm d}A_M(X)\xrightarrow{\sim}{\rm d}A_L(X)$ equals multiplication by 7 on $\Z/12\Z$.
\end{prop}
\begin{proof}
Recall that $\varphi_A$ acts as the identity map on $H^4(X,\Z)_{alg}$, thus $\varphi_A(\widetilde{H}_M)=\widetilde{H}_M$. By Lemma \ref{baseChg} we have
\[
\varphi_A(\widetilde{H}_M)=36L^2-17\widetilde{H}_L+24(\widetilde{F}_1+\widetilde{F}_2+\widetilde{F}_3)+12(Q_1+Q_2+Q_3).
\]
as a map from $A_M(X)$ to $A_L(X)$. Therefore
\[\begin{array}{ccl}
{\varphi_A}_*(-\frac{1}{12}\widetilde{H}_M)&=&-3L^2+\frac{17}{12}\widetilde{H}_L-2(\widetilde{F}_1+\widetilde{F}_2+\widetilde{F}_3)-(Q_1+Q_2+Q_3)\\
&=&-17\cdot(-\frac{1}{12}\widetilde{H}_L)\quad\mbox{mod}\;A_L(X)\\
&=&7\cdot(-\frac{1}{12}\widetilde{H}_L)\quad\mbox{mod}\;A_L(X).
\end{array}\]
\end{proof}

\noindent{\bf Remark.}
By the symmetry the Cremona transformation $f$,
the rank-8 lattice $H^4(X,\Z)_{alg}$ is also spanned by the classes
\[
	\{M^2, \widetilde{H}_M, \widetilde{G}_1, \widetilde{G}_2, \widetilde{G}_3, K_1, K_2, K_3\}
\]
constructed in a similar way from the right-hand side.
Here $\widetilde{G}_1, \widetilde{G}_2, \widetilde{G}_3$ are from the exceptional curves and $K_1, K_2, K_3$ are from the transverse double points.
The full transformation between the two set of bases is
\[
	\left(
	\begin{array}{c}
		M^2\\
		\widetilde{H}_M\\
		\widetilde{G}_1\\
		\widetilde{G}_2\\
		\widetilde{G}_3\\
		K_1\\
		K_2\\
		K_3
	\end{array}
	\right)
	=
	\left(
	\def\arraystretch{1.06}
	\begin{array}{cccccccc}
		7&-3&4&4&4&2&2&2\\
		36&-17&24&24&24&12&12&12\\
		4&-2&3&3&3&2&1&1\\
		4&-2&3&3&3&1&2&1\\
		4&-2&3&3&3&1&1&2\\
		2&-1&2&1&1&1&1&1\\
		2&-1&1&2&1&1&1&1\\
		2&-1&1&1&2&1&1&1
	\end{array}
	\right)\cdot
	\left(
	\begin{array}{c}
		L^2\\
		\widetilde{H}_L\\
		\widetilde{F}_1\\
		\widetilde{F}_2\\
		\widetilde{F}_3\\
		Q_1\\
		Q_2\\
		Q_3
	\end{array}
	\right)
\]
This expression is unique up to the ordering of the exceptional curves and the transverse double points on each side.
The top two rows are computed by Lemma \ref{baseChg}.
The other rows can be computed in a similar way.

\subsection{Proofs of Theorem~\ref{thm:DE} and its Corollary}
We first prove the theorem.

The derived equivalence follows from Proposition~\ref{deEquiv}. Note that this implies that the Picard numbers of $R_L$ and $R_M$ are the same.

Assume $R_L$ and $R_M$ have Picard number one. Suppose they are isomorphic. Then there is an isometry
\[
	\theta:T(R_L)\;\tilde{\rightarrow}\;T(R_M)
\]
which induces the isomorphism
\[\begin{array}{cccc}
\theta_*:&{\rm d}T(R_L)&\tilde{\rightarrow}&{\rm d}T(R_M)\\
&-\frac{\widetilde{H}_L}{12}&\mapsto&-\frac{\widetilde{H}_M}{12}
\end{array}\]
under the identifications ${\rm d}T(R_L)\simeq{\rm d}A_L(X)$ and ${\rm d}T(R_M)\simeq{\rm d}A_M(X)$.

By Proposition~\ref{times7}, the composition $\varphi_A\circ\theta$ is an automorphism on $T(R_L)$ acting as multiplication by 7 on ${\rm d}T(R_L)$. This contradicts the fact that the only automorphism on $T(R_L)$ is the identity 
\cite{Ogu02}. Hence $R_L$ and $R_M$ can't be isomorphic to each other.

Next we prove the corollary.

The corollary is trivial if $R_L$ and $R_M$ are isomorphic, so we assume that they are non-isomorphic.

Given a generic triple of points $\Pi_L\in R_L^{[3]}$, we determine a degree 12 K3 surface $R_M$ and a triple of points $\Pi_M\in R_M^{[3]}$ through the following steps:
\begin{enumerate}
\item Project $R_L$ from $\Pi_L$ to obtain $S_L\subset\lP^4$, whose ideal defines a Cremona transformation $f:\lP^4\dashrightarrow\lP^4$.
\item The base locus of $f^{-1}$ is a surface $S_M$ singular along three transverse double points. Normalize $S_M$ to get $\Sigma_M$.
\item $\Sigma_M$ is the blowup of a degree 12 K3 surface $R_M$ along three points. The three exceptional curves on $\Sigma_M$ are contracted to $\Pi_M\in R_M^{[3]}$.
\end{enumerate}

Recall that a pair of derived equivalent K3 surfaces of degree 12 uniquely determines each other up to isomorphism. So $R_M$ is independent of the choice of $\Pi_L\in R_L^{[3]}$ by Theorem~\ref{thm:DE}. Hence there is a rational map
\[\begin{array}{cccc}
\sigma:&R_L^{[3]}&\dashrightarrow&R_M^{[3]}\\
&\Pi_L&\mapsto&\Pi_M.
\end{array}\]
It is birational because $\Pi_L$ is uniquely determined by $\Pi_M$ through the same process as above.

\subsection{Connections between our construction and other approaches}
The derived equivalence and geometric connections
between the degree 12 K3 surfaces
$(R_L,H_L)$ and $(R_M,H_M)$ admit several interpretations.

\subsubsection{Mukai lattices}
For a K3 surface $R$, the {\em Mukai lattice} 
$$\widetilde{H}(R,\Z):= H^0(R,\Z) \oplus H^2(R,\Z) \oplus H^4(R,\Z),$$
equipped with a weight-two Hodge structure, i.e., the standard Hodge
structure on the middle summand and the outer summands taken
as $(1,1)$ classes. This is polarized by
$$(r_1,D_1,s_1)\cdot (r_2,D_2,s_2)= D_1 \cdot D_2 - r_1\cdot s_2 - r_2\cdot s_1.$$
Each coherent sheaf $E$ yields a Mukai vector
$$v(E)=(r(E),c_1(E),s(E)),$$
where $r(E)$ is the rank and $r(E)+s(E)=\chi(E)$.
Mukai \cite{Muk87} has shown that the second cohomology of a moduli space
$M_v(R)$ may be expressed
$$H^2(M_v(R),\Z)= 
\begin{cases} v^{\perp} & \text{ if } v\cdot v \ge 2 \\
	      v^{\perp}/\Z v & \text{ if } v\cdot v =0
\end{cases}
$$
provided $v$ is primitive and satisfies certain technical conditions.	
A derived equivalence between $R$ and $\hat{R}$ induces an isomorphism 
of Hodge structures
$$\Phi: 
\widetilde{H}(R,\Z) \stackrel{\sim}{\rightarrow}
\widetilde{H}(\hat{R},\Z)$$
which may be chosen so that $\hat{R}=M_v(R)$ with $\Phi(v)=(0,0,1)$.

We return to our degree 12 K3 surfaces $R_M$ and $R_L$.
We may interpret $R_M$ as a moduli space of vector bundles
on $R_L$ and {\em vice versa} \cite{Muk99}. Let $M_{(2,H_L,3)}(R_L)$
denote the moduli space of rank-two stable bundles $E$ with $c_1(E)=H_L$
and $\chi(R_L,E)=5$, which is isomorphic to $R_M$. The universal 
bundle $\mathcal{E} \rightarrow R_L \times R_M$ induces a Hodge isometry 
$$\Phi:
\widetilde{H}(R_L,\Z) \stackrel{\sim}{\rightarrow} \widetilde{H}(R_M,\Z)$$
described above. We have
$$
\Phi(2,H_L,3)=(0,0,1), \quad \Phi(0,0,1)=(2,H_M,3)$$
and $\Phi$ restricts to the isogeny on transcendental cohomology
mentioned in \S \ref{subsect:DES}. It follows formally
that
$$
\Phi(1,0,-2)=(-1,0,2),$$
thus after a shift the Mukai vector of ideal sheaves of length-three
subschemes of $R_L$ goes to the Mukai vector of length-three
subschemes of $R_M$. We obtain an isomorphism
$$H^2(R^{[3]}_L,\Z) \simeq H^2(R^{[3]}_M,\Z)$$
of Hodge structure arising from Mukai lattices.
Thus the Torelli Theorem \cite[Cor.~9.9]{Mar11} yields
a birational equivalence 
$$R_L^{[3]} \stackrel{\sim}{\dashrightarrow} R_M^{[3]}.$$
Corollary~\ref{cor:hilbthree} is quite natural from this perspective.

\medskip
\noindent {\bf Remark:} We also have $\Phi(1,0,-1)=(1,H_M,5)$. Elements
of $M_{(1,H_M,5)}(R_M)$ may be interpreted as $I_Z(H_M)$ where
$Z \subset R_M$ has length two. Similar reasoning gives 
$$R_L^{[2]} \stackrel{\sim}{\dashrightarrow} R_M^{[2]}.$$

\subsubsection{Homological projective duality}
Mukai \cite[Ex.~1.3]{Muk99} proposed an interpretation of the derived
equivalence via linear algebra, which was explored in detail by 
Iliev and Markushevich \cite{IM04}. The key observation is that the
components $\mathcal{S}^+$ and $\mathcal{S}^-$ described in 
\S \ref{subsect:OG} are naturally embedded in dual projective spaces
$$\mathcal{S}^+ \subset \lP^{15}, \quad
\mathcal{S}^- \subset \check{\lP}^{15}.$$
Codimension $r$ subspaces $\lP \subset \lP^{15}$
correspond to codimension $(16-r)$ subspaces
$\lP^{\perp}  \subset \check{\lP}^{15}$. When $r=8$,
the K3 surfaces $R=\mathcal{S}^+ \cap \lP$ and
$\check{R}=\mathcal{S}^- \cap \lP^{\perp}$ are dual.
These are derived equivalent
and generally non-isomorphic.
Kuznetsov \cite[\S 6.2]{Kuz06} has interpreted this derived equivalence via Homological Projective Duality.

It would be interesting to construct the Cremona transformation through this mechanism, by introducing the data of the three points on the K3 surface
into the duality construction.

\section{Zero divisors in the Grothendieck ring}\label{sect:Groth}
Let $K_0(Var/\C)$ denote the Grothendieck ring  of complex algebraic varieties. It is the abelian group generated by isomorphism classes of complex algebraic varieties subject to the relation
\[
	[Z]=[U]+[Z-U]
\]
where $U$ is an open subvariety of $Z$. The multiplication is induced by the Cartesian product:
\[
	[X][Y]=[X\times Y]
\]
which is associative and commutative with unit $1=[{\rm Spec}\,\C]$. More generally, if $Z\rightarrow X$ is a Zariski locally trivial bundle with fibers isomorphic to $Y$, by stratifying the base it's easy to prove that
\[
	[X][Y] = [Z].
\]

Let $\varmathbb{L}=[\varmathbb{A}^1]$ be the class of the affine line in $K_0(Var/\C)$. Consider a pair of non-isomorphic smooth projective varieties $X$ and $Y$ which are derived equivalent. It is interesting to know if there exists $k\geq0$ satisfying
\begin{equation}\label{L-equiv}
	([X]-[Y])\varmathbb{L}^k=0
\end{equation}
and what the minimal $k$ is if it exists \cite{KS16}.

When $X$ is a generic K3 surface of degree 12, Ito, Miura, Okawa and Ueda \cite{IMOU16K3} proves that there exists $Y$ non-trivially derived equivalent to $X$ such that (\ref{L-equiv}) holds for $k=3$. Actually, it can be improved to $k=1$ straightforwardly from the point of view of the Cremona transformation.

\begin{thm}
Let $R_L$ and $R_M$ be a generic pair of K3 surfaces associated with our Cremona transformation. Then we have
\[
	([R_L]-[R_M])\varmathbb{L}=0.
\]
in $K_0(Var/\C)$. The relation is minimal in the sense that $[R_L]-[R_M]\neq0$.
\end{thm}
\begin{proof}
Recall that $\Sigma_L$ is the normalization of $S_L$ as well as the blowup of $R_L$ at three points. Hence we have
\[
	[S_L] = [\Sigma_L]-3 = [R_L]+3\varmathbb{L}-3.
\]
From the blowup $\pi_1:X\rightarrow\lP^4$ we obtain
\begin{equation}\label{grothL}
\begin{array}{ccl}
	[X] &=& ([\lP^4]-[S_L])+[\pi_1^{-1}(S_L)]\\
	&=& ([\lP^4]-[S_L])+(([S_L]-3)[\lP^1]+[Q_1]+[Q_2]+[Q_3])\\
	&=& ([\lP^4]-[S_L])+(([S_L]-3)[\lP^1]+3[\lP^1]^2)\\
	&=& [\lP^4]+3[\lP^1]([\lP^1]-1)+[S_L]([\lP^1]-1)\\
	&=& [\lP^4]+3[\lP^1]\varmathbb{L}+[S_L]\varmathbb{L}\\
	&=& [\lP^4]+3[\lP^1]\varmathbb{L}+[R_L]\varmathbb{L}+3\varmathbb{L}^2-3\varmathbb{L}.
\end{array}
\end{equation}
By symmetry, we also have
\begin{equation}\label{grothM}
	[X] = [\lP^4]+3[\lP^1]\varmathbb{L}+[R_M]\varmathbb{L}+3\varmathbb{L}^2-3\varmathbb{L}.
\end{equation}
Subtracting (\ref{grothM}) from (\ref{grothL}) we get
\[
	([R_L]-[R_M])\varmathbb{L}=0.
\]

Next we show that $[R_L]\neq[R_M]$, and it is sufficient to show that $[R_L]\neq[R_M]$ modulo $\varmathbb{L}$. According to \cite{LL03}, $[R_L]=[R_M]\mod\varmathbb{L}$ if and only if $R_L$ and $R_M$ are stably birational. Because a K3 surface is not rationally connected, this implies that $R_L$ and $R_M$ are birational and thus isomorphic, contradicting Theorem~\ref{thm:DE}.
\end{proof}

\section{Exclusion of alternative constructions}\label{sect:exclude}
This section shows that there exists just one class of Cremona transformations of $\lP^4$ that can be resolved by blowing up an irreducible surface $S$ with transverse double points, i.e., the class constructed in Section~\ref{sect:const}. Recall that \cite{CK89} classified the case where $S$ is smooth.

\begin{thm}\label{thm:uniqueness}
Let $S\subset \lP^4$ be an irreducible surface of degree $d$ with $\delta>0$ transverse double points. Assume there exists a Cremona transformation
\[
	f: \lP^4 \dashrightarrow \lP^4
\]
resolved by blowing up $S$. Let $n$ and $\xi$ denote the degrees of the homogeneous forms
inducing $f$ and $f^{-1}$ respectively, and $m$ the multiplicity of $S$ in the base locus. Then we have
\[
	n=\xi=4,\quad m=1,\quad\delta = 3,
\]
and $S$ is obtained by projecting a degree $12$ K3 surface from three points.
\end{thm}

The remainder of this section is devoted to the proof of Theorem~\ref{thm:uniqueness}.

\subsection{Extracting Diophantine equations}
By Lemmas \ref{intEEi} and \ref{intLE}, equation (\ref{LM3=xi}) can be expressed as
\begin{equation}\label{xi=LM3}
\xi = n^3-3nm^2d+m^3\left(K_\Sigma C+5d\right).
\end{equation}
Similarly, equation (\ref{M4=1}) can be expressed as
\begin{equation}\label{1=M4}
1=n^4-6n^2m^2d+4nm^3(K_\Sigma C+5d)-m^4(15d+5K_\Sigma C+c_2(\Sigma)-6\delta)
\end{equation}
and equivalently as
\begin{equation}\tag{\ref{1=M4}'}\label{1=M4'}
1=n^4-6n^2m^2d+4nm^3(K_\Sigma C+5d)+m^4(d^2-25d-10K_\Sigma C-{K_\Sigma}^2+4\delta).
\end{equation}
The two formulas follow from the two expressions Lemma~\ref{intLE} (\ref{intLE3}) and (\ref{intLE3'}) for $E^4$, respectively. The right-hand sides of these equations are arranged as polynomials in $n$ and $m$. Note that only the coefficients of $m^4$ reflect the appearance of transverse double points.

\subsection{Enumeration of combinatorial cases}
\begin{lemma}\label{dioTab}
Only the following $(n,m,\xi)$ can occur.
\begin{center}\begin{tabular}{c|ccc}
&$n$&$m$&$\xi$\\
\hline
\rm (a)& $3$ & $1$ & $2$\\
\rm (b)& $4$ & $1$ & $4$\\
\rm (c)& $7$ & $2$ & $3$\\
\rm (d)& $9$ & $2$ & $9$\\
\rm (e)& $43$& $10$ & $7$\\
\rm (f)& $24$ & $5$ & $24$\\
\rm (g)& $49$ & $10$ & $49$
\end{tabular}\end{center}
\end{lemma}
\begin{proof}
In the smooth case, the same list \cite[Theorem 1.6]{CK89} is obtained by using \cite[Lemma 0.2]{CK89} and \cite[Formulae 0.3]{CK89}. 
The proof of the former proceeds unchanged even with the transverse double points. 
The latter can be derived from (\ref{xi=LM3}) and (\ref{1=M4}) 
and only the terms with power of $m$ up to two matter, so transverse double points don't change the result. Therefore the same elimination process works and we obtain the same list.
\end{proof}

\subsection{Exclusion of cases}\label{subsect:exclude}
Here we show that only Case (b) can occur.

\begin{lemma}
Cases (c) and (e) do not occur.
\end{lemma}
\begin{proof}
The proof is similar to the smooth case \cite[Lemma 3.2]{CK89}.

Assume Case (c) holds. Then (\ref{xi=LM3}) reduces to
\[
	2K_\Sigma C=11d-85
\]
and (\ref{1=M4}) reduces to
\[
	465=62d-2c_2(\Sigma)+12\delta.
\]
This is odd on the left and even on the right, a contradiction.

Assume Case (e) holds. Now (\ref{xi=LM3}) reduces to
\[
	79d=795+10K_\Sigma C,
\]
so $d$ is divisible by 5. On the other hand, (\ref{1=M4}) becomes
\[
	-34188=-11094d+1720(K_\Sigma C+5d)-100(15d+5K_\Sigma C+c_2(\Sigma)-6\delta).
\]
Note that 5 divides the right but not the left, a contradiction.
\end{proof}

\begin{lemma}
Cases (d), (f) and (g) do not occur.
\end{lemma}
\begin{proof}
Let $I_S$ be the ideal sheaf of $S\subset\lP^4$. Generally, the global sections of $I_S^m(n)$ and $\fO_{P'}(M)$ are bijective canonically. So we have
\begin{equation}\label{h0Ideal}
	h^0(\lP^4,I_S^m(n))=h^0(P',M)=5
\end{equation}
by equation (\ref{h0M=5}).

We prove the lemma case by case. In each case, we prove by contradiction in the following situations
\[
	h^0(\lP^4,I_S(4))=0,\;=1\;{\rm and}\;\geq2.
\]

Assume Case (d) holds.

Suppose $h^0(\lP^4,I_S(4))=0$. Consider the surjective map
\begin{equation}\label{surject_d}
	\bigoplus_{k_1+k_2=9} H^0(\lP^4,I_S(k_1))\otimes H^0(\lP^4,I_S(k_2))\twoheadrightarrow H^0(\lP^4,I_S^2(9)).
\end{equation}
By hypothesis $h^0(\lP^4,I_S(k))=0$ for all $k\leq4$. Since $k_1+k_2=9$ implies $k_1\leq4$ or $k_2\leq4$, the left-hand side of (\ref{surject_d}) vanishes. Thus $h^0(\lP^4,I_S^2(9))=0$, contradicting (\ref{h0Ideal}).

Let $X_0,...,X_4$ be a basis of degree one forms on $\lP^4$ in what follows.

Suppose $h^0(\lP^4,I_S(4))=1$. Let $A\in H^0(\lP^4,I_S(4))$ be a generator. This forces $h^0(\lP^4,I_S(k))=0$ for all $k\leq3$. It follows that $H^0(\lP^4,I_S^2(8))$ is generated by $A^2$.  Then (\ref{h0Ideal}) indicates that $A^2X_0,...,A^2X_4$ form a basis for $H^0(\lP^4,I_S^2(9))$. As a result, the linear system $|I_S^2(9)|$ defines an automorphism of $\lP^4$ instead of a Cremona transformation.

Suppose $h^0(\lP^4,I_S(4))\geq2$. Let $A,B\in H^0(\lP^4,I_S(4))$ be independent. Then $A^2$ and $AB$ are independent in $H^0(\lP^4,I_S^2(8))$. We claim that there exists an $i$ such that $A^2X_i$ is not a linear combination of $ABX_j$, $j=0,...,4$. Suppose not, i.e. $A^2X_i=ABL_i$ for some linear form $L_i$, $i=0,...,4$. Then we have $\frac{A}{B}=\frac{L_0}{X_0}=\frac{L_1}{X_1}$, which implies that $L_0=\frac{X_0L_1}{X_1}$, so $X_1$ divides $L_1$. Therefore $\frac{A}{B}=\frac{L_1}{X_1}$ is a scalar, thus $A$ and $B$ are dependent, a contradiction. As a result, there exists an $i$ such that $A^2X_i$ and $ABX_0,...,ABX_4$ form an independent subset of $H^0(\lP^4,I_S^2(9))$. Thus $h^0(\lP^4,I_S^2(9))\geq6>5$, a contradiction.

Assume Case (f) holds.

Suppose $h^0(\lP^4,I_S(4))=0$. Then $h^0(\lP^4,I_S(k))=0$ for all $k\leq4$. Now we consider the map
\begin{equation}\label{surject_f}
	\bigoplus_{k_1+\cdots+k_5=24} H^0(\lP^4,I_S(k_1))\otimes\cdots\otimes H^0(\lP^4,I_S(k_5))\twoheadrightarrow H^0(\lP^4,I_S^5(24)).
\end{equation}
At least one $k_i\leq4$, $i=1,...,5$, if their sum equals 24. Hence the left-hand side of (\ref{surject_f}) vanishes. Thus $h^0(\lP^4,I_S^2(24))=0\neq5$.

Suppose $h^0(\lP^4,I_S(4))\geq1$. Let $A\in H^0(\lP^4,I_S(4))$ be a nonzero element. Then $A^5\in H^0(\lP^4,I_S^5(20))$. Multiplication by $A^5$ defines an injection
\[
	\cdot A^5:H^0(\lP^4,\fO_{\lP^4}(4))\hookrightarrow H^0(\lP^4,I_S^5(24)).
\]
Thus $h^0(\lP^4,I_S^5(24))\geq{8\choose4}=70>5$, a contradiction.

The elimination of Case (g) is similar to Case (f). In Case (g), we use the surjection
\[
	\bigoplus_{k_1+\cdots+k_{10}=49} H^0(\lP^4,I_S(k_1))\otimes\cdots\otimes H^0(\lP^4,I_S(k_{10}))\twoheadrightarrow H^0(\lP^4,I_S^{10}(49))
\]
to rule out the situation $h^0(\lP^4,I_S(4))=0$. If $H^0(\lP^4,I_S(4))$ contains $A\neq0$, then multiplication of $A^{10}$ with 9-forms produces ${13\choose4}=715$ independent elements in $H^0(\lP^4,I_S^{10}(49))$, which is not allowed.
\end{proof}

\begin{lemma}\label{combined}
In cases (a) and (b) we have
\begin{center}\begin{tabular}{c|cc}
& \rm (a) $(3,1,2)$ & \rm (b) $(4,1,4)$\\
\hline
$d$ & $\leq8$ & $\leq15$\\
$K_\Sigma C$ & $4d-25$ & $7d-60$\\
$K_\Sigma^2$ & $d^2-11d+4\delta+30$ & $d^2+d+4\delta-105$\\
$c_2(\Sigma)$ & $19d-95+6\delta$ & $46d-405+6\delta$\\
$12\chi(\fO_\Sigma)$ & $d^2+8d-65+10\delta$ & $d^2+47d-510+10\delta$\\
$g(C)$ &$\frac{5d-23}{2}$ & $4d-29$
\end{tabular}\end{center}
The invariants $d$ and $\delta$ satisfy $(d-5)^2=2\delta$ in Case (a) and $(d-10)(d-15)=2\delta$ in Case (b) respectively.
\end{lemma}
\begin{proof}
In order to compute the invariants in the list, we first use (\ref{xi=LM3}) to express $K_\Sigma C$ in $d$ with given $n$, $m$ and $\xi$. Then (\ref{1=M4}) (resp. (\ref{1=M4'})) allows us to express $c_2(\Sigma)$ (resp. $K_\Sigma^2$) in $d$ and $\delta$. We compute $12\chi(\fO_\Sigma)$ and $g(C)$ by Noether's formula and the genus formula, respectively. The upper bound for $d$ comes from the inequality $d<(n/m)^2$ which holds generally \cite[Formulae 0.3 (v)]{CK89}.

We have $h^0(\lP^4,I_S(n))=h^0(P',M)=5$ by (\ref{h0M=5}). On the other hand, $h^1(\lP^4,I_S(n))=0$ by \cite[Prop. 7.1.4]{Dol12}. Hence
\[
h^0(\lP^4,I_S(n))=\chi(\lP^4,I_S(n))=\chi(P,I_{S'}(n))
\]
where the second equality follows from the functoriality of the Euler characteristic. The short exact sequence
\[
	0\rightarrow I_{S'}(nL-2\Sigma_iE_i)\rightarrow\fO_P(nL-2\Sigma_iE_i)\rightarrow\fO_{S'}(nC-2\Sigma_i(Q_i'+Q_i''))\rightarrow0
\]
implies that
\[
	\chi(P,I_{S'}(n))=\chi(P,nL-2\Sigma_iE_i)-\chi(S',nC-2\Sigma_i(Q_i'+Q_i'')).
\]
$\chi(P,nL-2\Sigma_iE_i)$ counts the dimension of the space of degree $n$ polynomials singular along $\Delta$, so
\[
	\chi(P,nL-2\Sigma_iE_i)= {n+4\choose4}-5\delta.
\]
By the previous computations and the Riemann-Roch formula, we have
\[\chi(S',nC-2\Sigma_i(Q_i'+Q_i''))=\bigg\{
\begin{array}{ll}
\frac{1}{12}(d^2-10d+385-62\delta)&\mbox{for (a)}\\
\frac{1}{12}(d^2-25d+930-62\delta)&\mbox{for (b)}
\end{array}\]
whence
\[\chi(P,I_{S'}(n))=\bigg\{
\begin{array}{ll}
-\frac{1}{12}(d^2-10d-35-2\delta)&\mbox{for (a)}\\
-\frac{1}{12}(d^2-25d+90-2\delta)&\mbox{for (b)}.
\end{array}\]
Then the two equations are obtained by setting $\chi(P,I_{S'}(n))=5$.
\end{proof}

\begin{lemma}
Case (a) does not occur.
\end{lemma}
\begin{proof}
Assume (a) is satisfied. Then the same argument as in \cite[Theorem 3.3]{CK89} implies that $d=5$. By Lemma \ref{combined} we have $\delta=0$.
\end{proof}

\subsection{Geometric analysis of the remaining case}
To complete the proof of Theorem~\ref{thm:uniqueness}, it remains to analyze the last possible case.

\begin{lemma}\label{subCaseb}
We have $(d,\delta) = (8,7)$ or $(9,3)$.
The invariants in these cases are
\[\begin{array}{c|ccccc}
(d,\delta) & K_\Sigma C & K_\Sigma^2 & c_2(\Sigma) & \chi(\fO_\Sigma) & g(C)\\
\hline
(8,7) & -4 & -5 & 5 & 0 & 3\\
(9,3) & 3 & -3 & 27 & 2 & 7
\end{array}\]
\end{lemma}
\begin{proof}
By the previous part only Case (b) is allowed.

By Lemma~\ref{combined},
we have $d\leq15$ and $g(C)=4d-29\geq 0$.
Hence $8\leq d\leq15$.
Then $(d-10)(d-15)=2\delta$ and our hypothesis $\delta>0$ force $d=8$ or $9$,
which implies that $\delta = 7$ or $3$, respectively.

The invariants are computed directly by using Lemma~\ref{combined}.
\end{proof}

Consider the linear system $|K_\Sigma+C|$ for both cases of Lemma \ref{subCaseb}.
We have $h^1(K_\Sigma+C)=0$ by Kodaira vanishing and $h^2(K_\Sigma+C) = h^0(-C)=0$ by Serre duality.

\begin{lemma}
The case $(d,\delta)=(8,7)$ is not allowed.
\end{lemma}
\begin{proof}
By the Riemann-Roch formula,
\[\begin{array}{ccl}
	h^0(K_\Sigma+C) &=& \chi(K_\Sigma+C)\\
	&=& \chi(\fO_\Sigma) + \frac{1}{2}(K_\Sigma+C)C\\
	&=&  0 +\frac{1}{2}(-4+8) = 2
\end{array}\]
Because $c_2(\Sigma) = 5$, $\Sigma$ can't be $\lP^2$, $\lP^1\times\lP^1$ or a minimal ruled surface. 
It implies that $\fO_\Sigma(K_\Sigma+C)$ is generated by global sections \cite[Prop. 2.2]{Som81}.
Hence the system $|K_\Sigma+C|$ defines a morphism
\[
	\phi: \Sigma\rightarrow\lP^1,
\]
the \emph{adjunction mapping}.

Consider the Stein factorization
\[\xymatrix{
	\Sigma\ar[dr]_r\ar[rr]^\phi&&\lP^1\\
	&\Sigma'\ar[ur]_s
}\]
where $r$ is a proper morphism with connected fibers and $s$ is a finite morphism.
By \cite[(2.3)]{Som81}, this leads to two possible situations:
\begin{enumerate}
	\item $\dim\phi(\Sigma) = 0$. Here we have
		$g(C)=1$, a contradiction.
	\item\label{adjImg1} $\dim\phi(\Sigma) = 1$. Then
		there exists a $\lP^1$-bundle $\pi: R\rightarrow\Sigma'$ such that $r$ factors as
		\[\xymatrix{
			\Sigma\ar[dr]_r\ar[r]^\epsilon&R\ar[d]\\
			&\Sigma',
		}\]
		where $\Sigma$ is the blowup of $R$ in at most one point of each fiber blown up,
		and $C$ meets the generic fiber with degree two.
		Furthermore, the map $s$ is an isomorphism except possibly if $g(C)=3$ and $h^{1,0}(\Sigma)=1$.
\end{enumerate}

Let's analyze Situation (\ref{adjImg1}):
The map $s$ can't be an isomorphism.
Otherwise, $R$ is a Hirzebruch surface
and $\chi(\fO_\Sigma) = \chi(\fO_R) = 1$, a contradiction.
Hence we obtain
\[
	g(\Sigma') = h^{1,0}(R) = h^{1,0}(\Sigma)=1.
\]
Then $\chi(\fO_\Sigma)=0$ and $c_2(\Sigma) = 5$ implies that $\Sigma$ has Hodge diamond
\[\begin{array}{ccccc}
&&1&&\\
&1&&1&\\
0&&7&&0.
\end{array}\]
Since the N\'eron-Severi group of $R$ has rank two,
we conclude that $\Sigma$ is the blowup of $R$ along five points on distinct fibers,
and $R$ is ruled over the elliptic curve $\Sigma'$.

Let $h$ be the class of a section on $R$ and $f$ be the class of a fiber so that
\[
	h^2 = m,\; hf=1,\;\mbox{and}\;f^2=0
\]
for some integer $m$.
According to the description of (\ref{adjImg1}),
the image of $C$ in $R$ gives a class $H = 2h + bf$ for some integer $b$ and $C = \epsilon^*H - \sum_{i=1}^5F_i$ where $F_1,...,F_5$ are the exceptional curves on $\Sigma$.
Note that $K_\Sigma = \epsilon^*K_R + \sum_{i=1}^5F_i$.
Thus we have
\[\begin{array}{ccc}
	8 = C^2 = H^2 -5 & \Rightarrow & H^2 = 13\\
	-4 = K_\Sigma C = K_R H + 5 & \Rightarrow & K_R H = -9
\end{array}\]
and consequently
\[\begin{array}{ccc}
	\chi(H) &=& \chi(\fO_R) + \frac{1}{2}H(H-K_R)\\
	&=& 0 + \frac{1}{2}(13+9) = 11.
\end{array}\]

On the other hand,
one can use the exact sequence
\[
	0 \rightarrow \fO_R \rightarrow \fO_R(h) \rightarrow \fO_h(m) \rightarrow 0
\]
to get $\chi(h) = m$,
and then use
\[
	0 \rightarrow \fO_R(h) \rightarrow \fO_R(2h) \rightarrow \fO_h(2m) \rightarrow 0
\]
to obtain $\chi(2h) = 3m$.
Then an induction on $n$ with the sequence
\[
	0 \rightarrow \fO_R(2h+(n-1)f) \rightarrow \fO_R(2h+nf) \rightarrow \fO_f(2) \rightarrow 0
\]
implies that $\chi(H) = \chi(2h+bf) = 3m+3b$.
But this implies $11 = \chi(H)$ is divisible by $3$, a contradiction.
\end{proof}

\begin{prop}
If $S$ has a transverse double point, then it can only be the image of a K3 surface $R\subset\lP^7$ of degree 12 projected from three points on $R$, and the number of transverse double points must be $\delta=3$.
\end{prop}
\begin{proof}
By the Riemann-Roch formula,
\[\begin{array}{ccl}
	h^0(K_\Sigma+C) &=& \chi(K_\Sigma+C)\\
	&=& \chi(\fO_\Sigma) + \frac{1}{2}(K_\Sigma+C)C\\
	&=&  2 +\frac{1}{2}(3+9) = 8.
\end{array}\]
Because $c_2(\Sigma) = 27$, $\Sigma$ can't be $\lP^2$, $\lP^1\times\lP^1$ or a minimal ruled surface.
It follows that $\fO_\Sigma(K_\Sigma+C)$ is generated by global sections \cite[Prop. 2.2]{Som81}.
Hence $|K_\Sigma+C|$ defines an adjunction morphism
with Stein factorization
\[\xymatrix{
	\Sigma\ar[dr]_r\ar[rr]^\phi&&\lP^7\\
	&\Sigma'\ar[ur]_s.
}\]
There are three possible situations
\cite[(2.3)]{Som81}:
\begin{enumerate}
	\item $\dim\phi(\Sigma) = 0$.
		We have $g(C)=1$, a contradiction.
	\item $\dim\phi(\Sigma) = 1$.
		Then $r:\Sigma\rightarrow\Sigma'$ is again obtained
	 by blowing up a $\lP^1$-bundle, with no more than one point in a fiber blown up.
		In particular, $1\geq\chi(\fO_\Sigma)=2$, a contradiction.
	\item\label{adjImg2} $\dim\phi(\Sigma) = 2$.
		Then $r:\Sigma\rightarrow\Sigma'$ expresses $\Sigma$ as the blowup of a smooth surface $\Sigma'$ along a finite set with $F\cdot C=1$ for any positive dimensional fiber $F$ of $r$.
		Moreover, $s:\Sigma'\rightarrow\lP^7$ is an embedding.
	\end{enumerate}

Now we are in Situation (\ref{adjImg2}).
Let $F_1,...,F_k$ be the exceptional curves on $\Sigma$ relative to $r$ and let $H$ be the very ample divisor on $\Sigma'$ which defines $s$.
Then
\[
	C = r^*H - \sum_{i=1}^kF_i\quad\mbox{and}\quad
	K_\Sigma = r^*K_{\Sigma'} + \sum_{i=1}^kF_i
\]
and it follows that
\[\begin{array}{c}
	9 = C^2 = H^2 - k\\
	3 = K_\Sigma C = K_{\Sigma'}H + k\\
	-3 = K_\Sigma^2 = K_{\Sigma'}^2 -k.
\end{array}\]
By the Riemann-Roch formula,
\[\begin{array}{ccl}
	8 = \chi(H) &=& \chi(\fO_{\Sigma'}) + \frac{1}{2}H(H-K_{\Sigma'})\\
	&=& 2 + \frac{1}{2}((9+k)-(3-k))\\
	&=& 5 + k,
\end{array}\]
which implies that $k=3$.
Hence $\Sigma$ is obtained by blowing up $\Sigma'$ along three distinct points,
and $\Sigma'\subset\lP^7$ has
\[
	\deg(\Sigma') = H^2 = 12,\quad K_{\Sigma'}H = K_{\Sigma'}^2 = 0,
	\quad c_2(\Sigma') = 24,\quad\chi(\fO_{\Sigma'}) = 2.
\]

We claim that $\Sigma'$ is a K3 surface.
Indeed, its Kodaira dimension $\kappa\neq2$ since $K_{\Sigma'}H = 0$.
If $\kappa = 1$,
then $\Sigma'$ has minimal model $R$ an elliptic surface,
such that $nK_R$ is numerically equivalent to a positive linear combination of some fiber classes if $n$ is large enough \cite[Prop. IX.3]{Bea96}.
This implies that $K_{\Sigma'}$ is numerically effective which contradicts to the fact that $K_{\Sigma'}H = 0$.
If $\kappa = -\infty$, then $h^{1,0}(\Sigma')=0$ and thus $1\geq\chi(\fO_{\Sigma'})=2$, a contradiction.
As a result,
$\Sigma'$ has $\kappa = 0$ and thus is a K3 surface.

Besides,
the birational map $R\dashrightarrow\Sigma\rightarrow S$ can be realized as the projection from three points on $R$.
Furthermore,
the fact that $\delta=3$ can also be verified directly by the double-point formula as in Section \ref{subsect:existence}.
\end{proof}

\bigskip
\bibliography{CremonaK3_bib}
\bibliographystyle{alpha}
\end{document}